\documentclass[10pt]{amsart}
\title{Extended Future Tube Conjecture for Unipotent Subgroups}
\author{Maxim Kukol}
\address{School of Mathematics and Natural Sciences, University of Wuppertal, Gaußstr. 20, 42119 Wuppertal, Germany}
\email{kukol@uni-wuppertal.de}

\usepackage[utf8]{inputenc}
\usepackage[english]{babel}

\usepackage{amsmath, amsfonts, amssymb, mathtools, mathdots, faktor, polynom, amsthm}

\usepackage{tabularx, caption, setspace, calc, multirow, bigdelim}

\usepackage{graphicx, pdfpages}
\usepackage{tikz}
\usepackage{tikz-cd}
\usetikzlibrary{arrows, fadings, patterns}

\usepackage{pgfplots}
\pgfplotsset{compat=1.18}

\usepackage{hyperref}
\usepackage[english]{cleveref}

\usepackage{extarrows}
\usepackage{enumerate}

\theoremstyle{plain} 
\newtheorem{Def}{Definition}[section]
\newtheorem{Lem}[Def]{Lemma}
\newtheorem{Coro}[Def]{Corollary}
\newtheorem{Prop}[Def]{Proposition}
\newtheorem{Theo}[Def]{Theorem}

\theoremstyle{definition} 

\newtheorem{Rk}[Def]{Remark}

\theoremstyle{plain} 

\newcommand{\Skip}{\mbox{}\\}
\newcommand{\NN}{\mathbb N}

\newcommand{\RR}{\mathbb R}
\newcommand{\CC}{\mathbb C}

\newcommand{\MM}{\mathcal M}
\newcommand{\sslash}{\mathbin{/\mkern-6mu/}}

\DeclareMathOperator{\re}{Re}
\DeclareMathOperator{\im}{Im}

\begin{document}
\maketitle
\begin{abstract}\vspace{-3em}\emergencystretch=1em
	Let $\Omega$ be the Lorentz future cone in $\RR^{d+1}$ with respect to the Lorentz product and let $T^M$ be the $M$-fold product of the future tube $T=\RR^{d+1}+i\Omega$. The Lorentz group $\textup{SO}_0(1,d)$ acts diagonally on $T^M$, and its complexification $\textup{SO}(1,d)^\CC$ acts on $\CC^{(d+1)\times M}$. We prove that the domain $G^\CC\cdot T^M$ is a Stein manifold for any connected unipotent subgroup $G$ of $\textup{SO}_0(1,d)$.
\end{abstract}
	
\section{Introduction}\label{sec1}
The future tube is the domain $T=\RR^{d+1}+i\Omega$ in $\CC^{d+1}=\RR^{d+1}+i\RR^{d+1}$, where
$$
\Omega=\left\{\omega=(\omega_0,\dots,\omega_d)\in\RR^{d+1}\mid \omega_0>0,\ \left\langle \omega,\omega\right\rangle_{1,d}>0\right\}
$$
is the Lorentz future cone for the Lorentz product
$$
\left\langle v,\tilde v\right\rangle_{1,d}:=v_0\tilde v_0-v_1\tilde v_1-\dots-v_d\tilde v_d,\ v,\tilde v\in\RR^{d+1}.
$$
Let $\textup{SO}_0(1,d)$ be the identity component of the group of real $(d+1)\times(d+1)$ matrices preserving the Lorentz product. This group acts properly by matrix multiplication on $T$, hence diagonally on the $M$-fold product $T^M$ for $M\in\NN$. Moreover, this diagonal action extends to an algebraic action of its complexification $\textup{SO}(1,d)^\CC$ on $\CC^{(d+1)\times M}$. The domain $\textup{SO}(1,d)^\CC\cdot T^M$ is called the extended future tube. In the theory of quantized fields, the extended future tube conjecture arose more than 60 years ago. It asserts that this tube is a domain of holomorphy. For more details and the physical background of this conjecture, see \cite{HW, J, SV, W}. The conjecture was first proved for $d=3$ in \cite{Z} (see also \cite{H}) and later for arbitrary $d$ in \cite{HS}. 

In this article we prove that the domain $Z:=G^\CC\cdot T^M$ is a Stein manifold for any connected unipotent subgroup $G$ of $\textup{SO}_0(1,d)$ (Corollary \ref{Coro:ZStein}). To this end we prove that the orbit space $Z/G^\CC$ is a geometric quotient and a Stein manifold (Theorem \ref{Theo:ZGStein}), from which the statement follows by a classical result of Matsushima and Morimoto, see \cite[Theorem 4]{M}. It was shown in \cite{HS} that there exist a complex analytic Hilbert quotient $\left(\textup{SO}(1,d)^\CC\cdot T^M\right)\sslash\textup{SO}(1,d)^\CC$ and an $\textup{SO}_0(1,d)$-invariant strictly plurisubharmonic function $\rho$ on $T^M$, which induces a strictly plurisubharmonic exhaustion function on this quotient by applying the minimum principle, see \cite{H, L}. A crucial role is played by the existence of an analytic Hilbert quotient $\CC^{(d+1)\times M}\sslash\textup{SO}(1,d)^\CC$ and the fact that the extended future tube is saturated with respect to this quotient. However, for an arbitrary closed subgroup $G$ of $\textup{SO}_0(1,d)$ it is not clear whether a Hilbert quotient exists. In Section \ref{sec2} we determine conditions under which the geometric quotient $Z/G^\CC$ exists without assuming the existence of an ambient quotient. In the following sections we prove that these conditions are met in our setting. It turns out that the minimum principle is also applicable and that $\rho$ induces a strictly plurisubharmonic function on the quotient. However, as we will show by a counterexample, this function is not exhaustive (Remark \ref{Rk:nonexhaustion}). Therefore, we will modify this function in order to obtain an exhaustion function.

\subsection*{Acknowledgements}
The author would like to thank Bernd Stratmann for many helpful discussions and for his significant contribution to the proof of Theorem \ref{Theo:existence}. 

\section{Existence of a geometric quotient}\label{sec2}
Let $G$ be a connected real form of a complex Lie group $G^\CC$ and $V$ a holomorphic Stein $G^\CC$-manifold. Further, let $X$ be a $G$-stable domain in $V$ on which the $G$-action is free and proper. Set $Z:=G^\CC\cdot X$. In this section we determine conditions under which the orbit space $Z/G^\CC$ is a geometric quotient. By this we mean that it is a complex space such that the quotient map $\pi:Z\longrightarrow Z/G^\CC$ is holomorphic and its structure sheaf is the sheaf of invariants. 
\begin{Def}\mbox{}
\begin{enumerate}[i)]
	\item We say that $X$ is orbit connected with respect to the $G$-action if the set 
	$$
	O_x:=\left\{g\in G^\CC\mid g\cdot x\in X\right\}
	$$ 
	is connected for all $x\in X$.
	\item We say that $X$ is weakly orbit connected with respect to the $G$-action if the local $G^\CC$-orbit 
	$$
	O(x):=G^\CC\cdot x\cap X
	$$
	is connected for all $x\in X$.
\end{enumerate}
\end{Def} 
Observe that orbit connectedness implies weak orbit connectedness. Assume that there exists a $G$-invariant strictly plurisubharmonic function $\rho:X\longrightarrow\RR$. 
\begin{Def}
	We say that $\rho$ is an exhaustion function mod $G$ along the local $G^\CC$-orbits if the restriction of the induced function $\overline{\rho}:X/G\longrightarrow \RR$ to $O(x)/G$ is an exhaustion function for all $x\in X$. 
\end{Def}
We denote by $\mathfrak{g}$ the Lie algebra of $G$. The function $\rho$ induces a momentum map $\mu:X\longrightarrow \mathfrak{g}^*$ with respect to the Kähler form $\omega=2i\partial\overline{\partial}\rho$ via 
$$
\mu^\xi(x):=\mu(x)(\xi):=\dfrac{d}{dt}\Big|_{t=0}\rho(\exp(it\xi)\cdot x).
$$
This means that $\mu$ is an equivariant map with respect to the coadjoint action satisfying $d\mu^\xi = \iota_{\xi_{X}}\omega$, where $\xi_{X}$ denotes the vector field on $X$ induced by $\xi$. We denote by $\MM$ the zero level set of $\mu$. By the equivariance of $\mu$, the group $G$ acts on $\MM$, and the quotient $\MM/G$ is the corresponding symplectic reduction of $X$. Moreover, the $G$-invariance of $\rho$ implies that
$$
\MM=\left\{x\in X\mid x\text{ is a critical point of }\rho|_{O(x)}\right\}.
$$
For $x\in X$, let 
$$
\mathfrak{g}\cdot x:=\left\{\xi_X(x)\mid \xi\in\mathfrak{g}\right\}=T_x(G\cdot x)\subset T_xX
$$ 
be the tangent space of the $G$-orbit at $x$. Since $G$ acts freely on $X$, the zero level set $\MM$ is a smooth
manifold. Moreover, for all $x\in \MM$, we have  
$$
\textup{Ker}(d\mu(x))=T_{x}\MM=(\mathfrak{g}\cdot x)^{\perp_{\omega}}=\mathfrak{g}\cdot x\oplus (\mathfrak{g}^\CC\cdot x)^\perp
$$
and the symplectic reduction of $X$ is a symplectic manifold, see \cite[p. 123]{MW}.
\begin{Prop}\label{Prop:Momentcap}
Assume that $X$ is weakly orbit connected with respect to the $G$-action and that $\rho$ is an exhaustion function mod $G$ along the local $G^\CC$-orbits. Then for all $x\in X$, we have $G^\CC\cdot x\cap\MM=G\cdot x_0$ for some $x_0\in\MM$.
\end{Prop} 
\begin{proof}
First, observe that for all $x\in X$ and all $\xi \in \mathfrak{g}\setminus\{0\}$, the function $t\mapsto \rho(\exp(it\xi)\cdot x)$ is strictly convex. This follows from the fact that $\rho$ is strictly plurisubharmonic and the $G$-action is free. Let $x \in X$. Since $\rho$ is an exhaustion function mod $G$ along the $G^\CC$-orbits, the restriction $\rho|_{O(x)}$ attains a minimum at some point $x_0 \in O(x)$. In particular, $x_0$ is a critical point of $\rho|_{O(x)}$, which implies that $x_0 \in \mathcal{M}$. Hence, $O(x) \cap \mathcal{M} \neq \emptyset$. A simple calculation shows that the critical set $O(x) \cap \mathcal{M}$ consists of a discrete set of $G$-orbits. Since every critical point is a local minimum (see \cite{H}, Proof of Lemma 2 in Sect. 2) and $O(x)$ is connected, it follows that $O(x) \cap \mathcal{M} = G\cdot x_0$.
\end{proof}
\begin{Theo}\label{Theo:existence}
	Assume that $X$ is weakly orbit connected with respect to the $G$-action and that $\rho$ is an exhaustion function mod $G$ along the local $G^\CC$-orbits. If $G^\CC$ acts freely on $Z$, the orbit space $Z/G^\CC$ is a complex manifold and a geometric quotient.
\end{Theo}
\begin{proof}
By the fundamental results in \cite[Thms. 12 and 24]{Hol}, it is enough to show that the $G^\CC$-action is proper. To this end let $(z_n)_{n\in\NN}\subset Z$ be a sequence in $Z$ and $(g_n)_{n\in\NN}$ a sequence in $G^\CC$ such that $(g_n\cdot z_n,z_n)\xrightarrow{n \to \infty}(z',z)$. Proposition \ref{Prop:Momentcap} implies that there exist $x\in G^\CC\cdot z\cap\MM$ and $x'\in G^\CC\cdot z'\cap\MM$. Let $g,g'\in G^\CC$ with $z=g^{-1}\cdot x$ and $z'=(g')^{-1}\cdot x'$. We define new sequences $x_n=g\cdot z_n$ and $\widetilde{g}_n=g'g_ng^{-1}$. We have $x_n\xrightarrow{n \to \infty} x$ and $\widetilde{g}_n\cdot x_n\xrightarrow{n \to \infty} x'$. For simplicity we write $g_n$ instead of $\widetilde{g}_n$. Without loss of generality, we may assume that $(x_n)_{n\in\NN}, (g_n\cdot x_n)_{n\in\NN}\subset X$. Let $V\subset\MM$ be a relatively compact open subset of $\MM$ containing $x$. Since $G$ acts freely on $X$, the momentum map $\mu$ is a submersion and we have 
$$
T_xX=T_{x}\MM\oplus i\mathfrak{g}\cdot x.
$$
It follows that there exists a relatively compact subset $U$ of $\exp(i\mathfrak{g})$ containing $e$ such that the map 
$$
\Phi:U\times V\longrightarrow X,\ (g,x)\mapsto g\cdot x
$$
is an open embedding. Let $U_0\subset U$ and $V_0\subset V$ be relatively compact neighborhoods of $e$ and $x$, respectively. In particular, there exists an $N\in\NN$ such that $x_n\in\Phi(U_0\times V_0)$ for all $n\geq N$. Without loss of generality, we may assume that $(x_n)_{n\in\NN}\subset\Phi(U_0\times V_0)$. Then there exists a sequence $(u_n)_{n\in\NN}$ in $U_0$ converging in $U$ such that $u_n\cdot x_n\in\MM$ for all $n\in\NN$. Applying the same argument to $(g_n\cdot x_n)_{n\in\NN}$ and $x'$, we may assume that $(x_n)_{n\in\NN}, (g_n\cdot x_n)_{n\in\NN}\subset\MM$ are converging sequences in $\MM$. Proposition \ref{Prop:Momentcap} implies that for all $n\in\NN$ there exists $h_n\in G$ such that $g_n\cdot x_n=h_n\cdot x_n$. Since $G^\mathbb{C}$ acts freely on $Z$, it follows that $g_n=h_n \in G$ for all $n \in \mathbb{N}$. The assertion then follows from the properness of the $G$-action.
\end{proof}
\begin{Coro}\label{Coro:iota}\mbox{}
	\begin{enumerate}[i)]
		\item The map $\overline{\iota}:\MM/G\longrightarrow Z/G^\CC$ induced by the inclusion $\iota:\MM\longrightarrow Z$ is a biholomorphism.
		\item The restriction $\rho|_{\MM}$ induces a smooth strictly plurisubharmonic function $\psi:Z/G^\CC\longrightarrow\RR$.
	\end{enumerate}
\end{Coro}
\begin{proof}
Since the momentum map $\mu$ is a submersion, we have
$$
\dim_\RR Z=\dim_\RR X=\dim_\RR\MM+\dim_\RR G. 
$$
Together with Proposition \ref{Prop:Momentcap} this implies that the map $\overline\iota$ is a diffeomorphism that induces the complex structure on the quotient $\MM/G$ from that on $Z/G^\CC$, cf. \cite[Sect. 2]{Kur}. Finally, applying the minimum principle (see \cite{H, L}), the induced strictly plurisubharmonic function $\psi$ is given by $\psi(q)=\inf\limits_{x\in\pi^{-1}(q)\cap X}\rho(x)$.
\end{proof}

\section{Orbit connectedness for unipotent subgroups of the Lorentz group}\label{sec3}
In this section we prove that for a connected unipotent subgroup $G$ of $\textup{SO}_0(1,d)$ the $M$-fold product $T^M$ of the future tube is orbit connected with respect to the $G$-action. 

Consider the Lie algebra of $\textup{SO}_0(1,d)$
$$
\mathfrak{so}(1,d)=\left\{\xi\in\textup{Mat}(d+1,\RR)\mid \left\langle \xi v,w\right\rangle_{1,d}=-\left\langle v,\xi w\right\rangle_{1,d}\text{ for all }v,w\in\RR^{d+1}\right\}.
$$ 
Since $\mathfrak{so}(1,d)$ is semisimple, there exists an Iwasawa decomposition 
$$
\mathfrak{so}(1,d)=\mathfrak{k}\oplus\mathfrak{a}\oplus\mathfrak{n},
$$
where $\mathfrak{n}\cong\RR^{d-1}$ is a nilpotent Lie algebra, and its non-zero elements are nilpotent of order $3$, see \cite[pp. 372-373]{Kn}. We recall that any two Iwasawa decompositions are conjugate in $\textup{SO}_0(1,d)$. The group $N=\exp(\mathfrak{n})$ is a unipotent abelian subgroup of $\textup{SO}_0(1,d)$. Its complexification is given by $N^\CC=N\exp(i\mathfrak{n})$ and is a unipotent complex algebraic group. Note that any connected unipotent subgroup $G$ of $\textup{SO}_0(1,d)$ is a closed subgroup of a conjugate of $N$. 

We choose suitable light-cone coordinates $(x^-,x^+,x')\in\CC\times\CC\times\CC^{d-1}$ on $\CC^{d+1}$ in which the extension of the Lorentz product to $\CC^{d+1}$ is given by 
$$
\left\langle x,\tilde x\right\rangle_{1,d}=\dfrac{1}{2}\left(x^-\tilde x^++x^+\tilde x^-\right)-\langle x',\tilde x'\rangle,
$$
where $\langle\cdot,\cdot\rangle$ denotes the extension of the Euclidean scalar product to $\CC^{d-1}$, such that $G$ is a subgroup of $N=\exp(\mathfrak{n})$, where the Lie algebra $\mathfrak{n}$ takes the form
$$
\mathfrak{n}=\left\{\begin{pmatrix}0&0&0^t\\0&0&2u^t\\u&0&0_{d-1}\end{pmatrix}\mid u\in\RR^{d-1}\right\}.
$$
In these coordinates, the Lorentz future cone is given by
$$
\Omega=\left\{\omega\in\RR^{d+1}\mid \omega^->0,\langle\omega,\omega\rangle_{1,d}>0\right\}.
$$
\begin{Lem}\label{Lem:product}Let $\xi\in\mathfrak{n}$ with $\xi\neq 0$.
	\begin{enumerate}[i)]
		\item For all $v\in \RR^{d+1}$ we have $\langle\xi v,\xi v\rangle_{1,d}\leq0$.
		\item For all $\omega\in\Omega$ we have $\langle\xi \omega,\xi \omega\rangle_{1,d}<0$.
	\end{enumerate}
\end{Lem}
\begin{proof}
Let $\xi\in\mathfrak{n}$ with $\xi\neq 0$. For $v=\left(v^-,v^+,v'\right)^t\in \RR^{d+1}$ a computation shows that
$$
\left\langle\xi v,\xi v\right\rangle_{1,d}=-(v^-)^2\|u\|^2\leq 0.
$$
For $\omega\in\Omega$, we have $\omega^->0$. Therefore, $\left\langle\xi\omega,\xi\omega\right\rangle_{1,d}<0$.
\end{proof}
\begin{Prop}\label{Prop:orbitcon}
	Let $G<\textup{SO}_0(1,d)$ be a connected unipotent subgroup. Then $T^M$ is orbit connected with respect to the $G$-action.
\end{Prop}
\begin{proof}
Let $x=(x_1,\dots,x_M)\in T^M$ with $x_j=v_j+i\omega_j$ and let
$$
O_x^{\im}:=\left\{\exp(i\xi)\in \exp(i\mathfrak{g})\mid \exp(i\xi)\cdot x\in T^M\right\}.
$$
Since $G^\CC=G\exp(i\mathfrak{g})$, we have $O_x=G\cdot O_x^{\im}$. As $G$ is connected, it suffices to show that $O_x^{\im}$ is connected. Let $\xi\in\mathfrak{g}$ with $\xi\neq 0$. Since $\xi$ is nilpotent of order $3$, we have 
$$
\exp(i\xi)\cdot x_j=x_j+i\xi x_j-\frac{1}{2}\xi^2x_j.
$$
We compute
\begin{align*}
	&\left\langle \im(\exp(i\xi)\cdot x_j),\im(\exp(i\xi)\cdot x_j)\right\rangle_{1,d}\\
	&=\left\langle\omega_j+\xi v_j-\frac{1}{2}\xi^2\omega_j,\omega_j+\xi v_j-\frac{1}{2}\xi^2\omega_j\right\rangle_{1,d}\\
	&=\left\langle\omega_j,\omega_j\right\rangle_{1,d}+2\left\langle\omega_j,\xi v_j\right\rangle_{1,d}+\left(\left\langle\xi\omega_j,\xi\omega_j\right\rangle_{1,d}+\left\langle\xi v_j,\xi v_j\right\rangle_{1,d}\right).
\end{align*}
By definition we have $\exp(i\xi)\cdot x_j\in T$ for all $j=1,\dots,M$ if and only if
$$
\left\langle \im(\exp(i\xi)\cdot x_j),\im(\exp(i\xi)\cdot x_j)\right\rangle_{1,d}>0.
$$
Rearranging this inequality yields
\begin{equation}\tag{$*$}\label{ineq}
-2\left\langle\omega_j,\xi v_j\right\rangle_{1,d}-\left\langle\xi\omega_j,\xi \omega_j\right\rangle_{1,d}-\left\langle\xi v_j,\xi v_j\right\rangle_{1,d}<\left\langle\omega_j,\omega_j\right\rangle_{1,d}.
\end{equation}
Lemma \ref{Lem:product} shows that the map 
$$
q_{x_j}:\mathfrak{g}\longrightarrow\RR,\ q_{x_j}(\xi)=-\left\langle\xi\omega_j,\xi \omega_j\right\rangle_{1,d}-\left\langle\xi v_j,\xi v_j\right\rangle_{1,d}
$$ 
is a positive quadratic form. Hence, it induces an inner product $g_{x_j}$ on $\mathfrak{g}$. Let $\|\cdot\|_{x_j}$ be the induced norm. Moreover, since the map $\xi\mapsto \left\langle\omega_j,\xi v_j\right\rangle_{1,d}$ is linear, there exists $\xi_{x_j}\in\mathfrak{g}$ such that $\left\langle\omega_j,\xi v_j\right\rangle_{1,d}=g_{x_j}(\xi,\xi_{x_j})$ for all $\xi\in\mathfrak{g}$. We get that (\ref{ineq}) is equivalent to
$$
\|\xi\|_{x_j}^2-2g_{x_j}(\xi,\xi_{x_j})<\left\langle\omega_j,\omega_j\right\rangle_{1,d}.
$$
Adding and subtracting $\|\xi_{x_j}\|_{x_j}^2$ shows that this is equivalent to
$$
\|\xi-\xi_{x_j}\|_{x_j}^2<\left\langle\omega_j,\omega_j\right\rangle_{1,d}+\|\xi_{x_j}\|_{x_j}^2.
$$
We set $R_{x_j}^2:=\left\langle\omega_j,\omega_j\right\rangle_{1,d}+\|\xi_{x_j}\|_{x_j}^2$ and $B_x:=\bigcap\limits_{j=1}^M B_{R_{x_j}}(\xi_{x_j})$, where $B_{R_{x_j}}(\xi_{x_j})$ denotes the ball with radius $R_{x_j}$ around $\xi_{x_j}$ in $\mathfrak{g}$. It follows that $\exp(i\xi)\cdot x\in T^M$ is equivalent to $\xi\in B_x$. Thus, $O^{\im}_x=\exp(iB_x)$, which is therefore connected.
\end{proof}

\section{The $\textup{SO}_0(1,d)$-invariant strictly plurisubharmonic function}\label{sec4}
The characteristic function of the Lorentz future cone is defined up to a constant by $\varphi(\omega)=\langle\omega,\omega\rangle_{1,d}^{-\frac{d+1}{2}}$. Its main properties are that for every sequence $(\omega_n)_{n\in\NN}\subset\Omega$ converging to a boundary point $\omega_0\in\partial\Omega$ we have $\lim\limits_{n\to\infty}\varphi(\omega_n)=\infty$ and that $\log\varphi$ is an $\textup{SO}_0(1,d)$-invariant strictly convex function, see \cite[pp. 10-11]{F}.  The Bergman kernel of the future tube $T$ is up to a positive constant given by $K(x,x)=c\varphi(2\im x)^2$, see \cite[p. 177]{F}. In particular,
$$
\rho:T^M\longrightarrow\RR,\ \rho(x_1,\dots,x_M)=\sum\limits_{j=1}^M\frac{1}{\langle\im(x_j),\im(x_j)\rangle_{1,d}}
$$
is an $\textup{SO}_0(1,d)$-invariant strictly plurisubharmonic function. In this section we prove the following proposition.
\begin{Prop}\label{Prop:exhaustionmodG}
Let $G<\textup{SO}_0(1,d)$ be a connected unipotent subgroup. Then $\rho$ is an exhaustion function mod $G$ along the local $G^\CC$-orbits.
\end{Prop}
\begin{proof}
Let $\pi:T^M\longrightarrow T^M/G$ be the quotient map, $x=(x_1,\dots,x_M)\in T^M$ and $r\in\RR$. Since $G^\CC=G\exp(i\mathfrak{g})$, we have 
$$
\pi(O(x)\cap\rho^{-1}(-\infty,r])=\pi(\exp(i\mathfrak{g})\cdot x\cap \rho^{-1}(-\infty,r]).
$$
The proposition will follow if we prove that $\exp(i\mathfrak{g})\cdot x\cap \rho^{-1}(-\infty,r]$ is compact. First, note that $\rho^{-1}(-\infty,r]\neq\emptyset$ if and only if $r>0$. Consider the map
$$
\beta_x:\mathfrak{g}\longrightarrow\CC^{(d+1)\times M},\ \beta_x(\xi)=\exp(i\xi)\cdot x.
$$
The properties of the characteristic function imply that $\rho^{-1}(-\infty,r]$ is closed in $\CC^{(d+1)\times M}$. Since $\beta_x$ is continuous, the set $D:=\beta_x^{-1}(\rho^{-1}(-\infty,r])$ is closed. We have $\beta_x(D)=\left(\exp(i\mathfrak{g})\cdot x\right)\cap\rho^{-1}(-\infty,r]$, so the statement will follow once we show that $D$ is bounded. For $\xi\in D$ and for all $j=1,\dots, M$, we have 
$$
\left\langle \im(\exp(i\xi)\cdot x_j),\im(\exp(i\xi)\cdot x_j)\right\rangle_{1,d}>\frac{1}{r}.
$$
Setting $x_j=v_j+i\omega_j$, the same computation as in the proof of Proposition \ref{Prop:orbitcon} yields
$$
\|\xi-\xi_{x_j}\|_{x_j}^2<\left\langle\omega_j,\omega_j\right\rangle_{1,d}+\|\xi_{x_j}\|_{x_j}^2-\frac{1}{r}.
$$
We set $R_{x_j,r}^2:=\left\langle\omega_j,\omega_j\right\rangle_{1,d}+\|\xi_{x_j}\|_{x_j}^2-\frac{1}{r}$ and $B_{x,r}:=\bigcap\limits_{j=1}^M B_{R_{x_j,r}}(\xi_{x_j})$. This shows that $D\subset B_{x,r}$, implying that it is bounded.
\end{proof}

\section{Kähler reduction for unipotent subgroups of the Lorentz group}\label{sec5}
Let $\rho$ be the $\textup{SO}_0(1,d)$-invariant strictly plurisubharmonic function introduced in Section \ref{sec4} and $\mu_{\textup{SO}_0(1,d)}:T^M\longrightarrow \mathfrak{so}(1,d)^*$ the induced momentum map (cf. Section \ref{sec2}). Let $G<\textup{SO}_0(1,d)$ be a connected unipotent subgroup and $Z:=G^\CC\cdot T^M$. By restriction the momentum map $\mu_{\textup{SO}_0(1,d)}$ induces a momentum map $\mu:T^M\longrightarrow\mathfrak{g}^*$. In this section we prove that $Z/G^\CC$ is a Stein manifold. This will in turn imply that $Z$ is Stein. 

Let $(x^-,x^+,x')\in\CC\times\CC\times\CC^{d-1}$ be the light-cone coordinates on $\CC^{d+1}$ introduced in Section \ref{sec3}. We associate $G^\CC$ with a $\CC$-linear subspace of $\CC^{d-1}\cong \mathfrak{n}^\CC$. The action of $G^\CC$ on $\CC^{d+1}$ in these coordinates is given by
$$
g\cdot(x^-,x^+,x')=(x^-,x^++2\langle x',g\rangle+x^-\langle g,g\rangle ,x'+x^-g).
$$
Moreover, for $x=(x_1,\dots,x_M)\in T^M$ with $x_j=v_j+i\omega_j$, the momentum map is given by 
$$
\mu^{g}(x)=-2\left\langle g,\sum\limits_{j=1}^M\dfrac{\omega^{-}_{j}v_j'-v^{-}_{j}\omega_j'}{\langle\omega_j,\omega_j\rangle_{1,d}^2}\right\rangle.
$$
Let $\pi_{G}$ be the orthogonal projection onto $G$ with respect to the decomposition $\RR^{d-1}=G\oplus G^\perp$. Observe that the zero level set of $\mu$ is given by
$$
\MM=\left\{x\in T^M\mid \pi_{G}\left(\sum\limits_{j=1}^M\dfrac{\omega^{-}_{j}v_j'-v^{-}_{j}\omega_j'}{\langle\omega_j,\omega_j\rangle_{1,d}^2}\right)=0\right\}.
$$
\begin{Lem}\label{Lem:isotropyN}\Skip
The group $N^\CC$ acts freely on $N^\CC\cdot T^M$.
\end{Lem}
\begin{proof}
Let $x=(x_1,\dots,x_M)\in T^M$ and $g\in N^\CC$. Then $g\cdot x_j=x_j$ implies that $x_j^-g=0$. Since $\im x_j\in\Omega$, it follows that $x_j^-\neq 0$ for all $j=1,\dots, M$ and therefore $g=0$.
\end{proof}
Now, Theorem \ref{Theo:existence} and Corollary \ref{Coro:iota} imply the following corollary. 
\begin{Coro}
	The orbit space $Z/G^\CC$ is a complex manifold biholomorphic to the symplectic reduction $\MM/G$ of $T^M$.
\end{Coro}
\begin{Rk}\label{Rk:nonexhaustion}
	The induced strictly plurisubharmonic function $\overline\rho:\MM/G\longrightarrow\RR$ is not exhaustive. For arbitrary $M$, consider the sequence
	$$
	x_n=((n+i,i,0'),(i,i,0'),(i,i,0'),\dots,(i,i,0'))\in T^M.
	$$
	Let $r\geq M$ and $K_r:=\rho^{-1}(-\infty,r]\cap\MM$. It is straightforward to check that $x_n\in K_r$ for all $n\in\NN$. However, we have for all $g\in G$,
	$$
	\|g\cdot x_n\|^2\geq\|g\cdot x_{n,1}\|^2\geq|(g\cdot x_{n,1})^-|^2=|x_{n,1}^-|^2=n^2+1
	$$
	thus $\inf\limits_{g\in G}\|g\cdot x_n\|\xrightarrow{n \to \infty}\infty$, yielding a counterexample to the exhaustion property of $\overline{\rho}$.
\end{Rk}
\begin{Theo}\label{Theo:ZGStein}	
	The orbit space $Z/G^\CC$ is a Stein manifold.
\end{Theo}
\begin{proof}
Since $Z/G^\CC$ is biholomorphic to $\MM/G$, we will prove that $\MM/G$ is Stein. Consider the function $\psi:Z\longrightarrow\RR$ defined by
\begin{align*}
\psi(z)=&\sum\limits_{j=1}^M \left(|z_j^-|^2+\frac{1}{|z_j^-|^2}+|\langle z_j,z_j\rangle_{1,d}|^2\right)+\sum\limits_{j=2}^M\|z_1^-z_j'-z_j^-z_1'\|^2\\
&+\|\pi_{(G^\CC)^\perp}(z_1')\|^2,
\end{align*}
where $\pi_{(G^\CC)^\perp}$ is the orthogonal projection onto $(G^\CC)^\perp$ with respect to the decomposition $\CC^{d-1}=G^\CC\oplus (G^\CC)^\perp$. Note that $\psi$ is well-defined on $Z$, since for all $z=g\cdot x \in Z$, we have $z_j^-=(g \cdot x_j)^-=x_j^- \neq 0$ for all $j=1,\dots,M$. It is straightforward to check that $\psi$ is a $G^\CC$-invariant smooth plurisubharmonic function. Consider the adapted function $\Psi:=\rho+\psi$ on $T^M$. This function is strictly plurisubharmonic, and since $\psi$ is $G^\CC$-invariant, $\Psi$ induces the same momentum map $\mu$ as $\rho$. Let $\overline{\Psi}$ be the induced strictly plurisubharmonic function on $\MM/G$. We have to prove that $\overline{\Psi}$ is an exhaustion function. To this end let $\pi:\MM\longrightarrow\MM/G$ be the quotient map and $r\in\RR$. Further, let $(\pi(x_n))_{n\in\NN}\subset\overline{\Psi}^{-1}(-\infty,r]$ be a sequence. We set 
$$
K_r:=\pi^{-1}\left(\overline{\Psi}^{-1}(-\infty,r]\right)=\Psi^{-1}(-\infty,r]\cap\MM
$$ 
and consider the preimage sequence $(x_n)_{n\in\NN}\subset K_r$. Since $\Psi$ and $\mu$ are smooth, $K_r$ is closed. We will prove that there exists a sequence $(g_n)_{n\in\NN}\subset G$ such that the sequence $(g_n\cdot x_n)_{n\in\NN}$ is bounded. This sequence would then have a convergent subsequence in $K_r$, which would imply that the sequence $(\pi(x_n))_{n\in\NN}$ has a convergent subsequence. Hence, $\overline{\Psi}^{-1}(-\infty,r]$ is compact.

Consider $\widetilde x_{n,j}':=\dfrac{1}{x_{n,j}^{-}}x_{n,j}'$ and $g_n=-\pi_{G}(\re(\widetilde x_{n,1}'))$.\\
i) Since $\rho>0$ and $\psi>0$, we have for all $j=1,\dots,M$, 
$$
|(g_n\cdot x_{n,j})^-|^2=|x_{n,j}^-|^2<\Psi(x_n)\leq r=:C_j^-(r)\text{ and }\frac{1}{|(g_n\cdot x_{n,j})^-|^2}=\frac{1}{|x_{n,j}^-|^2}<r.
$$
ii) We compute 
$$
\im(\widetilde x_{n,j}')=\frac{1}{|x_{n,j}^-|^2}\left(v_{n,j}^-\omega_{n,j}'-\omega_{n,j}^-v_{n,j}'\right).
$$
Observe that $g_n\cdot \widetilde x_{n,1}'=\widetilde x_{n,1}'-\pi_{G}(\re(\widetilde x_{n,1}'))$. It follows that
$$
\pi_{G^\CC}((g_n\cdot x_{n,1})')=x_{n,1}^{-}\pi_{G^\CC}(g_n\cdot \widetilde x_{n,1}')=ix_{n,1}^{-}\pi_{G}(\im(\widetilde x_{n,1}')),
$$
where $\pi_{G^\CC}$ is the orthogonal projection onto $G^\CC$. For $M=1$, $x_n\in\MM$ implies that
$$
0=\pi_{G}\left(\omega_{n,1}^-v_{n,1}'-v_{n,1}^-\omega_{n,1}'\right)=-|x_{n,1}^-|^2\pi_{G}\left(\im(\widetilde x_{n,1}')\right).
$$
Hence, $\pi_{G}(\im(\widetilde x_{n,1}'))=0$. It follows that
$$
\|(g_n\cdot x_{n,1})'\|^2=\|\pi_{(G^\CC)^\perp}((g_n\cdot x_{n,1})')\|^2=\|\pi_{(G^\CC)^\perp}(x_{n,1}')\|^2<\Psi(x_n)\leq r=:C_1'(r).
$$
For $M\geq 2$, we compute
\begin{align*}
\widetilde x_{n,1j}':=\widetilde x_{n,j}'-\widetilde x_{n,1}'&=\frac{1}{x_{n,1}^-x_{n,j}^-}\left(x_{n,1}^-x_{n,j}'-x_{n,j}^-x_{n,1}'\right).
\end{align*}
In particular, we have $g_n\cdot\widetilde x_{n,1j}'=\widetilde x_{n,1j}'$ and 
$$
\|\widetilde x_{n,1j}'\|^2<\frac{\Psi(x_n)}{|x_{n,1}^-|^2|x_{n,j}^-|^2}<r^3.
$$ 
Hence,
\begin{align*}
\|(g_n\cdot x_{n,j})'\|^2&=|x_{n,j}^-|^2\|g_n\cdot \widetilde x_{n,j}'\|^2\\
&<2r\left(\|\widetilde x_{n,1j}'\|^2+\|g_n\cdot \widetilde x_{n,1}'\|^2\right)\\
&<2r^4+2r\|g_n\cdot \widetilde x_{n,1}'\|^2.
\end{align*}
We set $\lambda_{n,j}:=\dfrac{-|x_{n,j}^-|^2}{\langle\omega_{n,j},\omega_{n,j}\rangle_{1,d}^2}$. Now, $x_n\in\MM$ implies that
\begin{align*}
0&=\pi_{G}\left(\sum\limits_{j=1}^M\dfrac{\omega_{n,j}^{-}v_{n,j}'-v_{n,j}^{-}\omega_{n,j}'}{\langle\omega_{n,j},\omega_{n,j}\rangle_{1,d}^2}\right)\\
&=\pi_{G}\left(\sum\limits_{j=1}^M\lambda_{n,j}\im(\widetilde x_{n,j}')\right)\\
&=\pi_{G}\left(\sum\limits_{j=1}^M\lambda_{n,j}\im(\widetilde x_{n,1j}'+\widetilde x_{n,1}')\right)\\
&=\sum\limits_{j=1}^M\lambda_{n,j}\pi_{G}(\im(\widetilde x_{n,1j}'))+\sum\limits_{j=1}^M\lambda_{n,j}\pi_{G}(\im(\widetilde x_{n,1}')).
\end{align*}
We set $c_{n,j}:=\dfrac{\lambda_{n,j}}{\sum\limits_{k=1}^M\lambda_{n,k}}$. Observe that $c_{n,j}>0$ and $\sum\limits_{j=1}^M c_{n,j}=1$. It follows that
\begin{align*}
\|\pi_{G}(\im(\widetilde x_{n,1}'))\|^2&=\left\|\sum\limits_{j=1}^M c_{n,j}\pi_{G}(\im(\widetilde x_{n,1j}'))\right\|^2\\
&\leq\sum\limits_{j=1}^M c_{n,j}\|\pi_{G}(\im(\widetilde x_{n,1j}'))\|^2\\
&\leq\sum\limits_{j=1}^M c_{n,j}\|\widetilde x_{n,1j}'\|^2<r^3.
\end{align*}
Finally, we have
\begin{align*}
\|g_n\cdot \widetilde x_{n,1}'\|^2&=\frac{1}{|x_{n,1}^-|^2}\|\pi_{G^\CC}((g_n\cdot x_{n,1})')+\pi_{(G^\CC)^\perp}((g_n\cdot x_{n,1})')\|^2\\
&\leq2\frac{1}{|x_{n,1}^-|^2}\left(\|\pi_{G^\CC}((g_n\cdot x_{n,1})')\|^2+\|\pi_{(G^\CC)^\perp}((g_n\cdot x_{n,1})')\|^2\right)\\
&<2\|\pi_{G}(\im(\widetilde x_{n,1}'))\|^2+2r^2\\
&<2(r^3+r^2).
\end{align*}
Hence, for all $j=2,\dots,M$, 
$$
\|(g_n\cdot x_{n,j})'\|^2<6r^4+4r^3=:C_j'(r).
$$
iii) For all $j=1,\dots, M$, we have $\langle x_{n,j},x_{n,j}\rangle_{1,d}=x_{n,j}^-x_{n,j}^+-\langle x_{n,j}',x_{n,j}'\rangle$. This implies that 
\begin{align*}
|(g_n\cdot x_{n,j})^+|^2&=\left|\dfrac{1}{x_{n,j}^-}\left(\langle x_{n,j},x_{n,j}\rangle_{1,d}+\langle (g_n\cdot x_{n,j})',(g_n\cdot x_{n,j})'\rangle\right)\right|^2\\
&<2r\left(\Psi(x_n)+\|(g_n\cdot x_{n,j})'\|^4\right)\\
&<2r^2+2rC_j'(r)^2=:C_j^+(r).
\end{align*}
In summary, for all $n\in\NN$, we have 
\begin{align*}
\|g_n\cdot x_n\|^2&=\sum\limits_{j=1}^M|(g_n\cdot x_{n,j})^-|^2+|(g_n\cdot x_{n,j})^+|^2+\|(g_n\cdot x_{n,j})'\|^2\\
&<\sum\limits_{j=1}^M\left(C_j^-(r)+C_j^+(r)+C_j'(r)\right).
\end{align*}
This shows that the sequence $(g_n\cdot x_n)_{n\in\NN}$ is bounded.
\end{proof}
\begin{Coro}\label{Coro:ZStein}
The domain $Z$ is a Stein manifold.
\end{Coro}
\begin{proof}
Since the group $G^\CC$ is a Stein manifold, the statement follows from a classical result of Matsushima and Morimoto, see \cite[Theorem 4]{M}.
\end{proof}

\bibliographystyle{amsalpha}
\bibliography{Extended_Future_Tube_Conjecture_for_Unipotent_Subgroups}

\end{document}